\documentclass[12pt]{amsart}
\usepackage{amsmath,amssymb}
\usepackage{amsfonts}
\usepackage{amsthm}
\usepackage{latexsym}
\usepackage{graphicx}

\def\codim{\mbox{codim}}

\def\R{\mathbb{R}}
\def\C{\mathbb{C}}

\def\vv<#1>{\langle#1\rangle}

\def\XXint#1#2{\setbox0=\hbox{$#1{#2}{\int}$}{#2}\kern-.5\wd0 }

\def\XXint#1#2#3{{\setbox0=\hbox{$#1{#2#3}{\int}$}
     \vcenter{\hbox{$#2#3$}}\kern-.5\wd0}}

\def\vv<#1>{\langle#1\rangle}
\newtheorem{thm}{Theorem}[section]

\newtheorem{lem}{Lemma}[section]

\newtheorem{cor}{Corollary}[section]
\theoremstyle{definition}

\theoremstyle{remark}

\numberwithin{equation}{section}

\begin{document}

\title{Nonpositively curved Hermitian metrics on product manifolds}


\begin{abstract} In this article, we classify all the Hermitian
metrics on a complex product manifold $M=X\times Y$ with nonpositive
 holomorphic bisectional curvature. It is a generalization of a result by Zheng
\cite{Zheng93}.
\end{abstract}

\renewcommand{\subjclassname}{\textup{2000} Mathematics Subject Classification}
\author{ Chengjie Yu}
\address{Department of Mathematics, Shantou University, Shantou, Guangdong, China } \email{cjyu@stu.edu.cn}

\date{Jan 2010}

\maketitle \markboth{ Chengjie Yu}
 {Nonpositively curved Hermitian metrics}
 In this article, using a trick as in Tam-Yu\cite{TY}, we prove the following generalization of the result theorem in Zheng
 \cite{Zheng93}.

 \begin{thm}\label{th-3}
Let $M=X\times Y$ with $X$ and $Y$ both compact complex manifolds.
Let $\phi_1,\phi_2,\cdots, \phi_r$ be a basis of $H^{1,0}(X)$ and
$\psi_1,\psi_2,\cdots,\psi_s$ be a basis of $H^{1,0}(Y)$. Then, for
any Hermitian metric $h$ on $M$ with nonpositive holomorphic
bisectional curvature,
\begin{equation*}
\omega_h=\pi_1^{*}\omega_{h_1}+\pi_2^{*}\omega_{h_2}+\rho+\bar\rho
\end{equation*}
where $h_1$ and $h_2$ are Hermitian metrics on $X$ and  $Y$ with
nonpositive holomorphic bisectional curvature respectively, $\pi_1$
and $\pi_2$ are natural projections from $M$ to $X$ and from $M$ to
$Y$ respectively, and
\begin{equation*}
\rho=\sqrt{-1}\sum_{k=1}^{r}\sum_{l=1}^{s}a_{kl}\phi_k\wedge\psi_l
\end{equation*}
with $a_{kl}$'s are complex numbers.
\end{thm}

Before the proof of Theorem \ref{th-3}, we need the following lemma.
\begin{lem}
Let $X^m$ and $Y^n$ be two compact complex manifolds. Let
$\phi_1,\phi_2,\cdots,\phi_{r}$ be a basis of $H^{1,0}(X)$ and
$\psi_1,\psi_2,\cdots,\psi_{s}$ be a basis of $H^{1,0}(Y)$. Let
$$
\rho=\rho_{ij}(x,y)dx^i\wedge dy^j
$$
be a global holomorphic two form on $X\times Y$, where
$(x^1,x^2,\cdots,x^m)$ is a local holomorphic coordinate of $X$, and
$(y^1,y^2,\cdots,y^n)$ is a local holomorphic coordinate of $Y$.
Then
\begin{equation}
\rho=\sum_{k=1}^{r}\sum_{l=1}^{s}a_{kl}\phi_k\wedge\psi_l
\end{equation}
where $a_{kl}$'s are complex numbers.
\end{lem}
\begin{proof}
Fix a local holomorphic coordinate $(y^1,y^2,\cdots,y^n)$ of $Y$, it
is clear that
\begin{equation*}
\theta_j=\sum_{i=1}^{m}\rho_{ij}(x,y)dx^i
\end{equation*}
is a global homomorphic 1-form on $X\times\{y\}$. Then
\begin{equation}
\theta_j=\sum_{k=1}^{r}b_{kj}(y)\phi_k
\end{equation}
where $b_{kj}$'s are local homomorphic functions on $Y$.

It is clear that
$$
\sum_{j=1}^{n}b_{kj}(y)dy^{j}
$$
is a global holomorphic 1-form on $Y$ for each $k$. So,
\begin{equation}
\sum_{j=1}^{n}b_{kj}(y)dy^{j}=\sum_{l=1}^{s}a_{kl}\psi_l
\end{equation}
where $a_{kl}$'s are complex numbers. Therefore
\begin{equation}
\rho=\sum_{j=1}^{n}\theta_j\wedge
dy^{j}=\sum_{k}^{r}\sum_{l=1}^{s}a_{kl}\phi_k\wedge\psi_l.
\end{equation}
\end{proof}

\begin{proof}[Proof of Theorem \ref{th-3}]
Let $(z^{m+1},\cdots,z^{m+n})$ be a local holomorphic coordinate of
$Y$ at $q$. Then, it is clear that
\begin{equation}
h_{\alpha\bar \alpha}(x,q)
\end{equation}
is a positive function on $X\times \{q\}$, where $m+1\leq\alpha\leq
m+n$.

Let $\Delta$ be the complex Laplacian on $X\times \{q\}$ and
$(z^1,z^2,\cdots,z^m)$ be a holomorphic coordinate of $X$ such that
$$h_{i\bar j}(x,q)=\delta_{i\bar j}$$
with $1\leq i, j\leq m$. Then
\begin{equation}
\Delta
h_{\alpha\bar\alpha}(x,q)=\sum_{i=1}^{m}\partial_{i}\partial_{\bar
i}h_{\alpha\bar\alpha}=-\sum_{i=1}^{m}R_{\alpha\bar\alpha i\bar
i}+\sum_{i=1}^mh^{\bar ba}\partial_ih_{\alpha\bar b}\partial_{\bar
i}h_{a\bar \alpha}\geq 0,
\end{equation}
with $1\leq a,b\leq n+m$. By maximum principle,
$h_{\alpha\bar\alpha}(x,q)$ is a constant function. Hence
\begin{equation}\label{eqn-d-i}
\partial_ih_{\alpha\bar b}=0.
\end{equation}
Interchange the roles of $X$ and $Y$ in the above, we get
\begin{equation}\label{eqn-d-alpha}
\partial_\alpha h_{i\bar b}=0.
\end{equation}
By the \eqref{eqn-d-i}, we know that
\begin{equation}
\partial_{i}h_{\alpha\bar\beta}=0
\end{equation}
for any $m+1\leq\alpha,\beta\leq n+m$. So, $h_{\alpha\bar\beta}$ is
independent of $z^{i}$'s. Then, $h_{\alpha\bar\beta}$ is a Hermitian
metric on $Y$. It is clear that $h_{\alpha\bar\beta}$ as a Hermitian
metric on $Y$ is of nonpositive holomorphic bisectional curvature
since holomorphic bisectional curvature deceases on complex
submanifolds. We denote this metric as $h_2$.

Similarly, by \eqref{eqn-d-alpha} $h_{i\bar j}$ is a Hermitian
metric on $X$ with nonpositive holomorphic bisectional curvature. We
denote it as $h_1$.

By \eqref{eqn-d-alpha} and \eqref{eqn-d-i}, we have
\begin{equation*}
\partial_{\alpha}h_{i\bar\beta}=0,\ \mbox{and}\ \partial_{\bar
i}h_{j\bar\alpha}=0.
\end{equation*}
So, the form $h_{i\bar\alpha}dz^i\wedge dz^{\bar\alpha}$ is a
holomorphic two form on $M_1\times \overline{M_2}$ where
$\overline{M_2}$ is the complex conjugate of $M_2$. By the lemma
above, we know that
\begin{equation}
h_{i\bar\alpha}dz^i\wedge
dz^{\bar\alpha}=\sum_{k=1}^{q_1}\sum_{l=1}^{q_2}a_{kl}\phi_k\wedge\bar\psi_l.
\end{equation}
Hence, we get the conclusion.
\end{proof}

The same as in Zheng \cite{Zheng93}, we have the following
consequence of the theorem.
\begin{cor}
$$\codim_\R(\mathcal{H}(M_1)\times
\mathcal{H}(M_2),\mathcal{H}(M_1\times M_2))=2h^{1,0}(M_1)\cdot
h^{1,0}(M_2)$$ where $M_1,M_2$ are compact complex manifolds, and
suppose that $\mathcal H(M_i)\neq\emptyset$ for $i=1,2$.
\end{cor}
\begin{proof} For any $h\in \mathcal H(M_1\times M_2)$, by the theorem, it has a
unique decomposition,
\begin{equation*}
\omega_h=\pi_1^{*}\omega_{h_1}+\pi_2^{*}\omega_{h_2}+\rho+\bar\rho
\end{equation*}
where $\rho=\sqrt{-1}\sum_{i=1}^{q_1}\sum_{j=1}^{q_2}a_{i
j}\phi_i\wedge\bar\psi_j$ with $a_{ij}\in \C$, $h_i\in \mathcal
H(M_i)$. So, we get a map
\begin{equation}
\mathcal H(M_1\times M_2)\to M(q_1\times q_2;\C),\ h\mapsto
(a_{ij})_{q_1\times q_2}.
\end{equation}
It is clear $\R^{+}$-linear. (Note that $\mathcal H(M_1\times M_2)$
is a convex cone.) So, it induce a linear map  of real vector
spaces,
\begin{equation*}
\Psi:\vv<\mathcal H(M_1\times M_2)>_\R\to M(q_1\times q_2;\C).
\end{equation*}
It is clear that
\begin{equation}
\ker \Psi=\vv<\mathcal H(M_1)\times\mathcal H(M_2)>_\R.
\end{equation}
Moreover, let $E_{k l}=(a_{ij})$ be such that $a_{i
j}=\delta_{ik}\delta_{jl}$. Note that
\begin{equation}
\begin{split}
&\pi_{1}^*\omega_{h_1}+\pi_{2}^{*}\omega_{h_2}+\sqrt{-1}(\phi_k+\psi_l)\wedge\overline{(\phi_k+\psi_l)}\\
=&[\pi_{1}^*\omega_{h_1}+\sqrt{-1}\phi_k\wedge\bar\phi_k]+[\pi_{1}^*\omega_{h_2}+\sqrt{-1}\psi_l\wedge\bar\psi_l]+\sqrt{-1}\phi_k\wedge\bar
\psi_l+\sqrt{-1}\psi_l\wedge\bar \phi_k.
\end{split}
\end{equation}
So, $E_{kl}$ is in the image of $\Psi$. Similarly, $\sqrt{-1}E_{kl}$
is also in the image of $\Psi$. Therefore, $\Psi$ is surjective. By
the dimension theorem in linear algebra, we get the identity.

\end{proof}

\end{document}